\documentclass[reqno]{amsart}

\usepackage{xspace,xcolor}
\usepackage{verbatim}
\usepackage{amsmath,amsthm,amssymb,amscd,latexsym}

\newtheorem{thm}{Theorem}[section]
\newtheorem{lem}[thm]{Lemma}
\newtheorem{cor}[thm]{Corollary}
\newtheorem{prop}[thm]{Proposition}

\newtheorem{rk}[thm]{Remark}
\newtheorem{df}[thm]{Definition}
\newtheorem{ex}[thm]{Example}

\numberwithin{equation}{section}

\providecommand{\AMS}{$\mathcal{A}$\kern-.1667em%
\lower.25em\hbox{$\mathcal{M}$}\kern-.125em$\mathcal{S}$}

\input xy
\xyoption{all}

\newcommand{\pr}[2]{\ensuremath{\langle {#1},{#2}\rangle}}
\newcommand{\norm}[1]{\ensuremath{\|#1\|}}

\newcommand{\cstar}{\mbox{$C^*$}}
\newcommand{\ma}{\ensuremath{\mathcal{A}}\xspace}
\newcommand{\mb}{\ensuremath{\mathcal{B}}\xspace}

\newcommand{\mh}{\ensuremath{\mathcal{H}}\xspace}
\newcommand{\mi}{\ensuremath{\mathcal{I}}\xspace}
\newcommand{\mj}{\ensuremath{\mathcal{J}}\xspace}
\newcommand{\mk}{\ensuremath{\mathcal{K}}\xspace}

\DeclareMathOperator{\gen}{span}

\DeclareMathOperator{\dom}{dom}

\newcommand{\ov}[1]{\overline{#1}}

\newcommand{\cspan}{\ov{\gen}}

\newenvironment{ntn}{\textbf{Notation.\,}}

\newcommand{\properideal}{%
  \mathrel{\ooalign{$\lneq$\cr\raise.22ex\hbox{$\lhd$}\cr}}}

\newcommand{\supp}{\mathrm{supp}}
\newcommand{\cmb}{C^*(\mb)}
\newcommand{\sss}{\scriptscriptstyle}
\newcommand{\lrangle}{\rangle_{\!{\scriptscriptstyle{L}}}}
\newcommand{\rrangle}{\rangle_{\!{\scriptscriptstyle{R}}}}
\newcommand{\prim}{\mathrm{Prim\thinspace}}
\newcommand{\ind}{\mathrm{Ind}}

\newcommand{\ZZ}{\mathbb{Z}}
\newcommand{\TT}{\mathbb{T}}

\newcommand{\cI}{\mathcal{I}}

\newcommand{\dmn}{D_{\mu\nu}^c}

\usepackage{todonotes}

\begin{document}

\title[Ideals in Cross Sectional \cstar-algebras]{Ideals in Cross Sectional \cstar-algebras of Fell Bundles}
\dedicatory{Dedicated to Marc Rieffel \\on the occasion of his seventy-fifth birthday}
\author[B. Abadie {\protect \and}F. Abadie]{Beatriz Abadie  {\protect
\and} Fernando Abadie}

\address{Centro de Matem\'atica. Facultad de Ciencias. Igu\'a 4225, CP 11 400, Montevideo, Uruguay.}
\email{abadie@cmat.edu.uy}
\address{Centro de Matem\'atica. Facultad de Ciencias. Igu\'a 4225, CP 11 400, Montevideo, Uruguay.}
\email{fabadie@cmat.edu.uy}

\begin{abstract} 
With each Fell bundle over a discrete group $G$ we associate a partial action of $G$ on
the spectrum of the unit fiber. We discuss the ideal structure of the corresponding full and reduced cross-sectional $C^*$-algebras in
terms of the dynamics of this partial action.  
\end{abstract}


\maketitle

\section*{Introduction}

The discussion of the ideal structure of crossed products by  a discrete group by means of the dynamical properties of the action  goes a long way back (see, for instance, \cite{eh},  \cite{zm}, \cite{od}).

 Archbold and Spielberg  discussed in \cite{ar} the relation between the ideal structure of the full crossed product and that of 
 the base algebra, under the assumption of topological freeness. More
 recently, the definition of topological freeness and several related
 results were extended to different settings: by Exel, Laca and Quigg for partial
 actions on commutative $C^*$-algebras 
 in \cite{elq},  by Lebedev in \cite{leb}, and later by Giordano and
 Sierakowski in \cite{gs}, for partial actions on
 arbitrary $C^*$-algebras, and by
 Kwa\'sniewski in \cite{kwa}) for crossed products by Hilbert $C^*$-bimodules.

We show  in this article that a Fell bundle $\mb$ over a discrete group $G$ gives rise to a partial action of $G$ on the spectrum of the unit fiber. This partial action agrees with those discussed in the works mentioned above, and we generalize to this context some of the results in them.

 This work is organized as follows. After establishing some background and notation in Section 1,
we introduce in Section 2 a partial action $\hat{\alpha}$ on the spectrum of the unit fiber of a Fell bundle \mb over a discrete group.  
When \mb is the Fell bundle corresponding to a partial action
$\gamma$, then $\hat{\alpha}$
agrees with $\hat\gamma$, as defined in \cite[Section~7]{fa1} or
\cite{leb}, and when $\mb$ is the Fell bundle associated in \cite{aee} with the crossed-product by a Hilbert $C^*$-bimodule, then $\hat{\alpha}$ is the homeomorphism $\hat{h}$ discussed in \cite{kwa}. 
Following familiar lines,  we establish in Section 3 a bijective correspondence between the family of  $\hat{\alpha}$-invariant open sets in the spectrum of the unit fiber and the set of ideals in $\mb$ (Proposition \ref{bundid} and Proposition \ref{alphainv})
This enables us to show that, when $\hat{\alpha}$ is topologically free, its minimality is equivalent to the simplicity of $C^*_r(\mb)$ (Corollary \ref{cor:2.6}). We then go on to generalize  to our setting, in Theorem \ref{thm:gs},  some of the results of Giordano and Sierakowski in \cite{gs} concerning the connection among the exactness property, the residual intersection property,  the structure ideal of $\mb$, and that of $C^*_r(\mb)$.

Finally, Section 4 contains some applications to the theory of Fell bundles with commutative unit fiber.

\section{Preliminaries}\label{sec:prel}
In this section we establish some notation and recall some basic definitions and facts about the spectrum of a \cstar-algebra and the Rieffel correspondence. We refer the reader 
to \cite{rw} for further details.

If $A$ is a \cstar-algebra, we denote by $\cI(A)$ the lattice of ideals in $A$ and by $\prim A$ the primitive space of $A$. That is, 
$\prim A$ is the set of primitive ideals with the hull-kernel topology. The spectrum of $A$, which we denote by $\hat{A}$, consists of the 
unitary equivalence classes of irreducible representations of $A$ with the initial topology for the map
\begin{equation}
\label{sptop}
 k:\hat{A}\longrightarrow \prim(A),\mbox{ given by } k([\pi])=\ker\pi \mbox{ for all } [\pi]\in\hat{A}.
\end{equation}
That is, a subset $S$ of $\hat{A}$ is open if and only if $S=k^{-1}(O)$, where $O$ is open in $\prim A$.
We will usually drop the brackets and denote $[\pi]\in \hat{A}$ by $\pi$.

Suppose now that $A$ and $B$  are \cstar-algebras and that $X$ is an $A-B$ imprimitivity bimodule. 
We denote by $\langle\ ,\ \lrangle$ and $\langle\ ,\ \rrangle$ the left and right inner products on $X$, respectively. 

An irreducible 
representation $\pi:B\rightarrow B(\mh_\pi)$ induces an irreducible representation $\ind_X\pi$ of $A$ as follows.
Let $X\otimes_B\mh_\pi$ be the Hilbert space obtained as the completion of the algebraic tensor product $X\odot_B\mh_\pi$ with respect
to the norm induced by the inner product determined by
\begin{equation}
\label{iptens}
\langle x\otimes h, y\otimes k\rangle:= \langle \pi(\langle y,x\rrangle)h,k\rangle,
\end{equation}
for $x,y\in X$ and $h,k\in \mh_\pi$.

Then $\ind_X\pi: A\longrightarrow B(X\otimes_B\mh_\pi)$ is defined by 
\begin{equation}
 \ind_X\pi(a)(x\otimes h)=ax\otimes h,
\end{equation}
for $a\in A$, $x\in X$, and $h\in \mh_\pi$.

Since $\ind_X\pi$ is irreducible as well, the imprimitivity bimodule $X$ yields a map
\begin{equation}
\label{ind}
 \ind_X: \hat{B}\longrightarrow \hat{A}
\end{equation}
that turns out to be a homeomorphism.

The imprimitivity bimodule $X$ also yields the Rieffel correspondence 
\[h_X:\cI(B)\longrightarrow \cI(A),\] which is a lattice isomorphism determined by the equation
\begin{equation}
 \label{rcor}
 h_X(I) X= XI,\text{ for all } I\in\cI(B),
\end{equation}
where 
\[XI=\cspan \{xi:x\in X, i\in I\} \text{ and } h_X(I)X=\cspan \{jx:x\in X, j\in h_X(I)\}. \]

These constructions are connected by the relation (\cite[3.24]{rw})
\begin{equation}
 \label{kerind}
 \ker \ind_X \pi=h_X(\ker \pi).
\end{equation}

If $J$ is an ideal in $A$, we denote by  $P_J$  the canonical projection on $A/J$. Let $X_J$ be the set
\begin{equation}
\label{xj}
 X_J=\{\pi\in \hat{A}: \pi|_J\neq 0\}.
\end{equation}
Then the map $J\mapsto X_J$ is a bijection from $\cI(A)$ onto the topology on $\hat{A}$. 

Besides, the maps
\[r_J: X_J\longrightarrow \hat{J} \text{ and } q_J: \hat{A}\setminus X_J\longrightarrow \widehat{A/J}, \] 
 determined, respectively,  by
 \begin{equation}
 \label{rj}
  r_J(\pi)=\pi|_J \text{ and }  q_J(\pi)\circ P_J=\pi
\end{equation}
are homeomorphisms.

If $X$ is an $A-B$ imprimitivity bimodule and $J$ is an ideal in $B$, then $XJ=h_X(J)X$, and $X/XJ$ is an $A/h_X(J)- B/J$ imprimitivity bimodule. Furthermore, the diagram
\begin{equation}
\label{diag}
\xymatrix{
\widehat{B/J}\ar[r]^{\ind_{{ X/XJ}}} & \widehat{A/h_X(J)}\\
\hat{B}\setminus X_J\ar[u]^{q_J} \ar[r]_{\ind_X} & \hat{A}\setminus X_{h_X(J)}\ar[u]_{q_{h_X(J)}}}
\end{equation}
commutes.

\section{The partial action associated with a Fell bundle}\label{sec:pafb}
\begin{ntn}
Throughout this work  $\mb=(B_t)_{t\in G}$ will denote a Fell bundle  over a discrete
group $G$. 
We will make use of the usual notation:
\[X^*\!=\{x^*:x\in X\}\subseteq B_{t^{-1}},\  X_1X_2\cdots X_n=\cspan\{x_1x_2\cdots x_n:x_i\in X_i\}
\subseteq B_{t_1t_2\cdots t_n},\]
for $X\subseteq B_t$ and  $X_i\subseteq B_{t_i}$, where $t, t_i\in G$ and $i=1,\cdots,n$.

In this setting, $B_t$ is a Hilbert \cstar-bimodule over $B_e$, for left 
and right multiplication and inner products given by
\begin{equation}
 \langle b_1,b_2\lrangle= b_1b^*_2,\ \langle b_1,b_2\rrangle=b_1^*b_2.
\end{equation}

We denote by $C^*(\mb)$ the cross-sectional $C^*$-algebra of $\mb$, and by $C_c(\mb)$  the dense *-subalgebra  of compactly supported 
cross sections.

The map $E:C_c(\mb)\longrightarrow B_e$ consisting of evaluation at $e$ extends to a conditional expectation $E:C^*(\mb)\longrightarrow B_e$.

We next recall some definitions and results related to the reduced cross-sectional $C^*$-algebra of a Fell bundle.
Further details and proofs can be found in \cite{examen}.

Let $\ell^2(\mb)$ denote the right Hilbert $C^*$-module over $B_e$ consisting of those sections $\xi$
such that $\sum_{t\in G} \xi^*(t)\xi(t)$ converges in $B_e$.

Thus, $\ell^2(\mb)$ is the direct sum of the right $B_e$-Hilbert $C^*$-modules $\{B_t:t\in G\}$. Let $j_t:B_t\longrightarrow \ell^2(\mb)$ be the inclusion map. That is,  
\begin{equation}
\label{jayt}
j_t(b)=b\delta_t,\text{ for }t\in G\text{ and }b\in B_t,
\end{equation}
 where  
$b\delta_t(s)=\delta_{s,t}b$,  $\delta_{s,t}$ being the Kronecker delta. Then $j_t$ is adjointable, and its adjoint is
evaluation at $t$.

\par Each $b_t\in B_t$ defines an adjointable operator
$\Lambda_{b_t} \in\mathcal{L}(\ell^2(\mb))$, given by
\[\Lambda_{b_t}(\xi)(s)=b_t\xi(t^{-1}s),\  \forall \xi\in\ell^2(\mb),\ 
s\in G.\]

The reduced $C^*$-algebra $C^*_r(\mb)$ of the Fell bundle
$\mb$ is the $C^*$-subalgebra of $\mathcal{L}(\ell^2(\mb))$ generated
by  $\{\Lambda_{b}: b\in \mb\}$. The correspondence $b_t\mapsto \Lambda_{b_t}$ extends to a *-homomorphism
 \[\Lambda: C^*(\mb)\longrightarrow C^*_r(\mb) \]
verifying (\cite[3.6]{examen})
\begin{equation}
\label{kerlam}
\ker \Lambda= \{c\in C^*(\mb): E(c^*c)=0\}.
\end{equation}
 
 We will often view  $B_e$ as a  $C^*$-subalgebra of $C^*_r(\mb)$ by identifying $a\in B_e$ with
$\Lambda_a\in C^*_r(\mb)$.

We denote by $D_t$ the ideal in $B_e$ defined by $D_t=B_tB^*_t$. Since the structure described 
above makes $B_t$ into a $D_t-D_{t^{-1}}$ imprimitivity bimodule, $B_t$ yields, as in equation (\ref{ind}),
a homeomorphism
\[\ind_{B_t}:\hat{D}_{t^{-1}}\longrightarrow \hat{D_t}.\]

We will denote by $X_t$, $r_t$, and  $q_t$, respectively, the set $X_{\sss{D}_t}$ and the maps $r_{\sss{D}_t}$ and $q_{\sss{D}_t}$ defined in 
(\ref{xj}) and (\ref{rj}). Notice that $X_e=\hat{B_e}$. Finally, we
denote by $\hat{\alpha}_t$ the homeomorphism that makes the
diagram \[
\xymatrix{X_{t^{-1}}\ar[d]_{r_{t^{-1}}}\ar[r]^{\hat{\alpha}_t}&X_t\ar[d]^{r_t}\\
            \hat{D}_{t^{-1}}\ar[r]_{\ind_{B_t}}&\hat{D}_{t} }
\]
commute. That is,  
\begin{equation}
\label{alpha}
\hat{\alpha}_t: X_{t^{-1}}\longrightarrow X_t\text{ is given by } \hat{\alpha}_t=r^{-1}_t\circ \ind_{B_t}\circ r_{t^{-1}},
\end{equation}
for all $t\in G$.
\end{ntn} 
\begin{rk}
\label{rkalpha}
 If $\pi\in X_{t^{-1}}$ is a representation of $D_e$ on $\mh_\pi$, then $\hat{\alpha}_t(\pi)$ is the representation of $D_e$ on $B_t\otimes_{D_{t^{-1}}} \mh_\pi$ given by
 \begin{equation}
  \label{ext}
  \big(\hat{\alpha}_t(\pi)a\big)(b\otimes h)=ab\otimes h,
 \end{equation}
for all $a\in D_e$, $b\in B_t$, and $h\in \mh_\pi$.
\end{rk}
\begin{proof}
When $a \in D_t$ the result follows straightforwardly from the definition, and equation (\ref{ext}) clearly defines an extension of 
$\ind_{B_t}(\pi|_{D_{t^{-1}}})$ to a representation of $D_e$. 
\end{proof}

\begin{prop}\label{prop:pafb}
 Given a Fell bundle $\mb=(B_t)_{t\in G}$  over a discrete group $G$, let $\hat{\alpha}_t$ be the homeomorphism defined in 
 equation (\ref{alpha}), for $t\in G$.

 Then $\hat{\alpha}:=\big(\{X_t\}_{t\in
G},\{\hat{\alpha}_t\}_{t\in G}\big)$ is a partial action of $G$ on $\hat{B_e}$.  
\end{prop}
\begin{proof}
Clearly, $\hat{\alpha}_t$ is a homeomorphism between open subsets of $X$, so it remains to show that $\hat{\alpha}_{st}$ extends 
$\hat{\alpha}_s\hat{\alpha}_t$, for all $s,t\in G$.

We first show that $\dom \hat{\alpha}_s\hat{\alpha}_t \subseteq \dom\hat{\alpha}_{st}$. Let $\pi\in \dom \hat{\alpha}_s\hat{\alpha}_t $, and assume that \mbox{$\pi\not\in \dom\hat{\alpha}_{st}$}. 
That is, $\pi|_{D_{(st)^{-1}}}=0$. We will show that this implies that  
 $\hat{\alpha}_t(\pi)|_{D_{s^{-1}}}= 0$, which contradicts the fact that $\pi\in \dom \hat{\alpha}_s\hat{\alpha}_t $.
 
 In fact, let $d\in D_{s^{-1}}$. Then, for $b\in B_t$ and $h\in \mh_\pi$, we have
 \begin{align*}
  \|\hat{\alpha}_t(\pi)(d)(b\otimes h)\|^2 &=\langle db\otimes h,db\otimes h\rangle\\
  &=\langle \pi(b^*d^*db)h,h\rangle\\
  &=0,
 \end{align*}
because $b^*d^*db\in B^*_tD_{s^{-1}}B_t=B^*_tB^*_sB_sB_t\subseteq B^*_{st}B_{st}=D_{(st)^{-1}}$.

We now show that $\hat{\alpha}_{st}=\hat{\alpha}_s\hat{\alpha}_t$ on  $\dom \hat{\alpha}_s\hat{\alpha}_t $.
Namely, we will show that if  $\pi\in \dom \hat{\alpha}_s\hat{\alpha}_t $ is a representation on $\mh_\pi$, then the map
\[U: B_s\otimes_{\sss{D_{s^{-1}}}}\!B_t\otimes_{\sss{D_{t^{-1}}}}\! \mh_\pi\longrightarrow B_{st}\otimes_{\sss{D_{(st)^{-1}}}}\! \mh_\pi,\]
defined by 
\[U(b_s\otimes b_t\otimes h)=b_sb_t \otimes h,\] 
for $b_s\in B_s$, $b_t\in B_t$, and $h\in \mh_\pi$, is a unitary 
operator intertwining $\hat{\alpha}_s\hat{\alpha}_t(\pi)$ and $\hat{\alpha}_{st}(\pi)$.

In order to check that the definition of $U$ makes sense, first notice that

\begin{align*}
 B_s\otimes_{\sss{D_{s^{-1}}}}\!B_t\otimes_{\sss{D_{t^{-1}}}}\! \mh_\pi&=B_s\otimes_{\sss{D_{s^{-1}}}}\!D_{s^{-1}}B_t\otimes_{\sss{D_{t^{-1}}}}\! \mh_\pi\\
&=B_s\otimes_{\sss{D_{s^{-1}}}}\! D_{s^{-1}}(B_tB_t^*B_t) \otimes_{\sss{D_{t^{-1}}}}\! \mh_\pi \\  
&=B_s\otimes_{\sss{D_{s^{-1}}}}\!B_t(B_t^*D_{s^{-1}}B_t)\otimes_{\sss{D_{t^{-1}}}}\! \mh_\pi.
\end{align*}
This implies that the map
\[\tilde{U}: B_s\times B_t\times\mh_\pi \longrightarrow B_{st}\otimes_{\sss{D_{(st)^{-1}}}}\! \mh_\pi\]
defined by $\tilde{U}(b_s,b_t,h)=b_sb_t\otimes b_{st}$  is balanced: given $b_s\in B_s$, $b_t\in B_t$, $e\in B_t^*D_{s^{-1}}B_t$, 
$c\in D_{t^{-1}}$,   
 and $h\in \mh_\pi$, we have that $ec \in B_t^*D_{s^{-1}}B_tD_{t^{-1}}=B_t^*D_{s^{-1}}B_t\subseteq D_{(st)^{-1}}$.

 Therefore,
 \begin{align*}
  \tilde{U}(b_s,b_tec,h)&=b_sb_tec\otimes h=b_sb_t\otimes \pi(ec)h\\
  &=b_sb_t\otimes \pi(e)\pi(c)h\\
  &=b_sb_te\otimes \pi(c)h\\
  &=\tilde{U}(b_s,b_t, \pi(c)h).
 \end{align*}

 Besides, $U$ is an isometry because if $b_s,c_s\in B_s$, $b_t,c_t\in B_t$, and $h,h'\in \mh_\pi$, then
 \begin{align*}
  \langle b_s\otimes b_t\otimes h, c_s\otimes c_t\otimes h'\rangle &=\langle \big(\hat{\alpha}_t(\pi)(c^*_sb_s)\big)(b_t\otimes h),
  c_t\otimes h'\rangle\\
  &= \langle c_s^*b_sb_t\otimes h, c_t\otimes h'\rangle\\
  &=\langle \pi(c_t^*c_s^*b_sb_t)h,h'\rangle\\
  &=\langle b_sb_t\otimes h, c_sc_t\otimes h'\rangle\\
  &=\langle U(b_s\otimes b_t\otimes h), U(c_s\otimes c_t\otimes h')\rangle.
  \end{align*}
Furthermore,  $U$ is onto because its image is a non-zero $\hat{\alpha}_{st}(\pi)$-invariant subspace of $B_{st}\otimes \mh$. Finally, it 
is apparent that $U$ intertwines 
$\hat{\alpha}_s\hat{\alpha}_t(\pi)$ and $\hat{\alpha}_{st}(\pi)$.
 \end{proof}
\begin{df}
 Let $\mb$ be a Fell bundle over a discrete group $G$. The partial action $\hat{\alpha}$ in Proposition \ref{prop:pafb} will be called the partial action associated 
 with $\mb$.
\end{df}
\begin{ex}\label{cphcb}{\rm{\bf Crossed products by Hilbert $C^*$-bimodules.}
 When $\mb$ is the Fell bundle associated to a Hilbert $C^*$-bimodule $X$ over a $C^*$-algebra $A$ as in \cite[2.6]{aee}, the associated partial action $\hat{\alpha}$ is the partial homeomorphism $\hat{h}$ 
 discussed in \cite{kwa}. When the $C^*$-algebra $A$ is commutative it also agrees with the partial homeomorphism 
 induced by the partial action $\theta$ in \cite[1.9]{ae1}.}
\end{ex}
\label{pcp}

\begin{ex}{\rm {\bf Partial crossed products.}
\par If $\gamma=(\{\gamma_t\}_{t\in G},\{D_t\}_{t\in G})$ 
 is a partial action of a discrete group $G$ on a $C^*$-algebra
 $A$, then the Fell bundle $\mb_\gamma$ associated with $\gamma$ has
 fibers $B_t=\{t\}\times D_t$ with the obvious structure of Banach
 space, and product and involution given by: 
\begin{gather*}
(r,d_r)(s,d_s)=(rs,\gamma_r(\gamma_{r^{-1}}(d_r)d_s)),\\
(r,d_r)^*=(r^{-1},\gamma_{r^{-1}}(d^*_r)).
\end{gather*}
The unit fiber of $\mb_\gamma$ gets identified with $A$ in the obvious way.

The partial action $\gamma$ induces a partial action  $\hat{\gamma}$ on $\hat{A}$ that was defined in \cite[\S 7]{fa1} and
\cite{fa2} and further discussed in \cite{leb}.   The partial action  $\hat{\gamma}$ is given by
\[\hat{\gamma}_t(\pi)=\pi\circ \gamma_{t^{-1}}\text{ for } \pi\in \hat{A},\]
and it agrees with the partial action associated with the Fell bundle $\mb_\gamma$. In fact, it is easily 
checked that, if $\pi\in \hat{D}_{t^{-1}}$ is a representation on a Hilbert space $\mh_\pi$, then the map
\[U:B_t\otimes_{D_{t^{-1}}}\mh_\pi\longrightarrow\mh_\pi\text{, determined by } U((t,d_t)\otimes h)=\pi(\gamma_{t^{-1}}(d_t))(h),\]
for $d_t\in D_t$, $t\in G$, and $h\in \mh_\pi$, is a unitary operator intertwining $\ind_{B_t}\pi$ and $\pi\circ\gamma_{t^{-1}}$.}

\end{ex}

\begin{ex}
\label{cuf}
{\rm  {\bf Fell bundles with commutative unit fiber}. We now assume that the Fell bundle \mb has commutative unit fiber, that is, 
$B_e=C_0(X)$, for a locally compact Hausdorff space $X$. 
We identify $X$ with $\hat{B_e}$ in the usual way: $x\in X$ is viewed as 
$[\pi_x]\in \hat{B_e}$, where $\pi_x$ is evaluation at $x$. 

If $I_x=\ker \pi_x$, then $x\in X_{t^{-1}}$ if and only if 
$B^*_tB_t\not \subseteq I_x$.
That is (\cite[3.3]{rw}), $x\in X_{t^{-1}}$ if and only if $B_t I_x\neq B_t$.

Therefore, if $b_t(x)$ denotes the image of an element $b_t$ of $B_t$ under the quotient map on $B_t/B_tI_x$, then
\[\hat{B_e}\setminus X_{t^{-1}}=\{x\in X: b_t(x)=0 \text{ for all }b_t\in B_t\}.\]

Besides, if $x\in X_{t^{-1}}$, we have,  by (\ref{kerind}),
\[I_{\hat{\alpha}_t(x)}B_t=I_{\hat{\alpha}_t(x)}D_tB_t=\ker(\ind_{B_t}\pi_x)B_t=B_t\ker \pi_x=B_tI_x.\]
Therefore,
\begin{equation}
 \label{abt}
(ab_t)(x)=\begin{cases}a(\hat{\alpha}_t(x))b_t(x)&\textrm{ if }x\in
X_{t^{-1}}\\ 0&\textrm{ otherwise,}\end{cases}
\end{equation}
for $a\in B_e$ and $b_t\in B_t$.}
\end{ex}

 \section{Topological Freeness and Ideals in the Cross-Sectional C*-algebras}\label{sec:rideals}

In this section we show that some well-known results relating topological freeness and the ideal structure of crossed products carry over to our setting.

\begin{prop}
 \label{htort}
 Let $\mb=(B_t)_{t\in G}$ be a Fell bundle over a discrete group $G$, and let $\rho$ be a representation of $\cmb$ on a Hilbert space 
 $\mk$. Suppose that $\sigma:B_e\longrightarrow B(\mh)$ is an irreducible subrepresentation of $\rho|_{B_e}$,  and let 
 $\mh_t=\cspan\  \rho(B_t)\mh$, for each $t\in G$. Then
 \begin{enumerate}
  \item  $\mh_t$ is $\rho(B_e)$-invariant for all $t\in G$.
  \item $\mh_t=\{0\}$ if $\sigma\not\in X_{t^{-1}}$, and $\mh_t\perp\mh$ if $\sigma\not\in X_{t}$.
  \item If $\sigma\in X_{t}\cap X_{t^{-1}}$ and $\hat{\alpha}_t(\sigma)\neq \sigma$, then  $\mh_t\perp\mh$.
  \end{enumerate}
\end{prop}
\begin{proof}
 Statement (i) is apparent. As for (ii),  consider the orthogonal decompositions
 \[\mk=\mh\oplus \mh^\perp,\  \rho|_{B_e}= \sigma \oplus \sigma^\perp.\]
Notice that any element in $B_t$ can be written as $xb_ty$, where $x\in D_t$, $b_t\in B_t$, and $y\in D_{t^{-1}}$. Besides, if $\sigma\not\in X_{t^{-1}}$, then $\sigma|_{D_{t^{-1}}}=0$, and, 
for any $h\in \mh$,
\[\rho(xb_ty)(h)=\rho(xb_t)(\sigma(y)(h)+\sigma^\perp(y)h)=0,\]
which shows that $\mh_t=\{0\}$.

If $\sigma \not \in X_t$, then, for $x,b_t,y$ as above, and $h,h'\in \mh$,
\begin{align*}
 \langle \rho(xb_ty)h,h'\rangle & =\langle \rho(b_ty)h, \rho(x^*)h'\rangle\\
 &=\langle \rho(b_ty)h, \sigma(x^*)h'+\sigma^\perp(x^*)h'\rangle\\
 &=\langle\rho(b_ty)h, \sigma(x^*)h'\rangle\\
 &=0,
 \end{align*}
which completes the proof of (ii).
In order to prove (iii), we now assume that $\sigma\in X_t\cap X_{t^{-1}}$. Let $\sigma_t$ denote the subrepresentation of 
$\rho|_{B_e}$ on $\mh_t$, that is, 
\[\sigma_t(c)h_t=\rho(c)h_t,\]
for all $c\in B_e$ and $h_t\in \mh_t$.
Then the map
\[U:B_t\otimes_{D_{t^{-1}}}\mh\longrightarrow \mh_t \text{ given by }U(b_t\otimes h)=\rho(b_t)h\]
is a unitary operator intertwining $\sigma_t$ and $\hat{\alpha}_t(\sigma)$. In fact, if $b_t, c_t\in B_t$, and $h,k\in \mh$, then
\[\langle b_t\otimes h, c_t\otimes k\rangle=\langle \sigma(c_t^*b_t)h,k\rangle =\langle \rho(c_t^*b_t)h,k\rangle =\langle \rho(b_t)h,\rho(c_t)k\rangle.\]

Therefore, if $\sigma \neq \hat{\alpha}_t(\sigma)$, then $\sigma$ and $\sigma _t$ are irreducible  non-equivalent subrepresentations of $\rho|_{B_e}$. It now follows from 
\cite[12.15]{antleb} that $\mh$ and $\mh_t$ are orthogonal.

\end{proof}

\begin{df}
 Recall from \cite[2.2]{elq} that a partial action $\theta$ of a discrete group $G$ on a locally compact topological space $X$ is  
  {\em topologically free} if for any finite subset $S$ of $G\setminus \{e\}$ the set 
  \[\bigcup_{t\in S} \{x\in \dom \theta_t:\theta_t(x)=x\}\]
has empty interior. Equivalently, $\theta$ is topologically free if 
  the set
 \[F_t=\{x\in \dom \theta_t:\theta_t(x)=x\}\]
 has empty interior for any $t$ in $G\setminus \{e\}$.
 
\end{df}

\begin{thm}\label{thm:topfree} Suppose that $\mb=(B_t)_{t\in G}$ is a Fell bundle over a discrete group $G$, $A$ is a $C^*$-algebra,
and
\[\phi: C^*(\mb)\longrightarrow A \]
 is a *-homomorphism, and let $J:=\ker\phi\cap B_e$.
 
 If the partial action $\hat{\alpha}$ associated with $\mb$ is topologically free on 
 $\hat{B_e}\setminus X_J$, then
  \begin{equation}\label{ineq}
    \norm{\phi(c)}\geq\norm{\phi(E(c))},\qquad \forall c\in
    C^*(\mb).\end{equation}  
    \end{thm}
\begin{proof}
 
 Since it suffices to show that (\ref{ineq}) holds when $c$ belongs to the dense \mbox{$\star$-subalgebra} $C_c(\mb)$ of compactly 
 supported 
 cross sections, we assume that 
\[c= \sum_{t\in \supp(c)} c(t)\delta_t,  \]
 where $\supp(c)$ is a finite subset of $G$.
 In order to show the statement, we will prove that
 \begin{equation}\label{ineqep}
    \norm{\phi(c)}\geq\norm{\phi(E(c))}-\epsilon,
    \end{equation} 
 for all $\epsilon >0$. 
 
 Fix $\epsilon >0.$
 Note  that
 \begin{gather*}
  \label{quotnorm}
  \|\phi(E(c))\|=\|E(c)+J\|_{B_e/J}=\max\{\|\tau (E(c)+J)\|:\tau\in \widehat{B_e/J}\}\\
  =\max\{\|\sigma (E(c))\|:\sigma\in \hat{B_e}\setminus X_J\}.
   \end{gather*}
   
   Besides, since the map $\sigma\mapsto \|\sigma (E(c))\|$ is lower semicontinuous on $\hat{B_e}\setminus X_J$ (\cite[A30]{rw}), 
   we can choose a set $V$  that  is open in $\hat{B_e}\setminus X_J$ and such that
   \begin{equation}
    \label{holdsonv}
    \|\sigma(E(c)\|\geq\|\phi(E(c))\|-\epsilon,
    \end{equation}
    for all $\sigma\in V$.

    Now, since $\hat{\alpha}$ is topologically free on $\hat{B_e}\setminus X_J$, the set
    \begin{equation}
    \label{F}
         F=\bigcup_{\substack{t\in \supp(c)\\ t\neq e}} \{\sigma \in X_{t^{-1}}: \hat{\alpha}_t(\sigma)=\sigma\} 
    \end{equation}   
    does not contain $V$. Thus, we can choose a representation $\sigma\in V$ on a Hilbert space $\mh$ 
    such that $\sigma\not\in F$.
    
    Let $\tilde{\phi}:B_e/J\rightarrow \phi(B_e)$ be the canonical
    isomorphism induced by $\phi|_{B_e} $, and let $\psi_0$ be a state
    of $\phi(B_e)$ 
    associated with the irreducible representation 
    $q_J(\sigma)\circ(\tilde{\phi})^{-1}$, where $q_J$ is as in (\ref{rj}). Extend $\psi_0$ to a pure state $\psi$ on $\phi(\cmb)$.  The GNS construction for $\psi$ 
    yields a representation $\pi$ 
    of $\phi(\cmb)$ on a Hilbert space $\mk$ containing a closed subspace $\mh$ such that 
    $q_J(\sigma)\circ(\tilde{\phi})^{-1}$ is the subrepresentation of $\pi|_{\phi(B_e)}$ on $\mh$.

    We now define $\rho:\cmb\longrightarrow B(K)$  by $\rho=\pi\circ \phi.$
    If $Q\in B(\mk,\mh)$ is the orthogonal projection on $\mh$, then
    \begin{equation}
     \label{qpiq}
     Q\rho(b)Q^*= Q\pi(\phi(b))Q^*=Q(\pi(\tilde{\phi}(b+J))Q^*= q_J(\sigma)(b+J)=\sigma(b),
     \end{equation}
for all $b\in B_e$, which shows that $\sigma$ is an irreducible subrepresentation of $\rho|_{B_e}$. 

We now set $\mh_t=\cspan \rho(B_t)(\mh)$. By Proposition 
\ref{htort} we have, since $\sigma\not\in F$, that $\mh_t \perp \mh$ for all $t\in \supp(c)$ such that $t\neq e$.

Therefore,
\begin{gather*}
 \|\phi(c)\|\geq \|\pi\circ\phi(c)\| =\|\rho(c)\|\geq\|Q\rho(c)Q^*\|
 =\|Q\rho(E(c))Q^*\|\\
 = \|\sigma(E(c))\|
 \geq \|\phi(E(c))\|-\epsilon.
\end{gather*}

\end{proof}

\begin{cor} 
\label{ibe}Suppose that $\mb=(B_t)_{t\in G}$ is a Fell bundle over a discrete group $G$ such that the partial action 
associated with $\mb$ is topologically free. Then
\begin{enumerate}
 \item If $I$ is an ideal in $C^*(\mb)$ such that $I\cap B_e=\{0\}$, then $I\subset \ker \Lambda$, where  
 \[\Lambda: C^*(\mb)\longrightarrow C^*_r(\mb) \]
 is the canonical surjective map. 
 \item If $I$ is an ideal in $C_r^*(\mb)$ such that $I\cap B_e=\{0\}$, then $I=\{0\}$. Consequently, 
 a representation of $C_r^*(\mb)$ is faithful if and only if its restriction to $B_e$ is faithful.
\end{enumerate} 
\end{cor}
\begin{proof}
(i) Since the restriction to $B_e$ of the quotient map $P_I:C^*(\mb)\longrightarrow C^*(\mb)/I$ is injective, we have, by Theorem 
\ref{thm:topfree}, that
\[\|P_I(E(c))\|\leq \|P_I(c)\|\text{, for all }c\in C^*(\mb).\]
Consequently, $E(I)\subseteq I\cap B_e=\{0\}$, and  $I\subset \ker \Lambda $ (see Equation (\ref{kerlam})).

(ii) Let $J=\Lambda^{-1}(I)$. Then $J \vartriangleleft C^*(B)$ and $\Lambda(J\cap B_e)\subseteq I\cap B_e=\{0\}$. 

Therefore, $J\cap B_e\subseteq \ker\Lambda \cap B_e=\{0\}.$
It now follows from (i) that $J\subseteq \ker \Lambda$.
Hence, $I=\Lambda(J)=\{0\}$.

\end{proof}

\begin{df}(cf. \cite{fa1}) Let $\mb$ be a Fell bundle over a discrete group $G$. A subset $\mj\subseteq\mb$ is  an {\em ideal} 
of \mb if it is a Fell bundle over $G$ with the inherited
structure, and if $\mj\mb=\mj=\mb\mj.$
 An ideal $I$ in $B_e$ is said to be {\em \mb-invariant} if  $B_tIB_t^*\subseteq I$, for all  $t\in G$. 
\end{df}

\begin{prop} 
\label{binv} Let $\mb$ be a Fell bundle over a discrete group $G$, and let $I$ be an ideal in $B_e$.  
Then the following statements are equivalent:
\begin{enumerate}
 \item $I$ is a $\mathcal{B}$-invariant ideal. 
 \item $B_tIB_t^*=I\cap B_tB_t^*$, $\forall t\in G$. 
\item $B_tI=IB_t$, $\forall t\in G$.
\item $\mathcal{I}=(IB_t)_{t\in G}$ is an ideal of $\mathcal{B}$.  
\end{enumerate}
\end{prop}
\begin{proof} 
Suppose that $I$ is $\mathcal{B}$-invariant. Then $B_tIB_t^*\subseteq I$, and,
since $B_tIB_t^*\subseteq B_tB_eB_t^*=B_tB_t^*$, we have that
$B_tIB_t^*\subseteq I\cap B_tB_t^*$. 

On the other hand, since
$B_t^*IB_t\subseteq I$, we have that  
\[I\cap B_tB_t^* =IB_tB_t^*=B_tB_t^*IB_tB_t^*\subseteq B_tIB_t^*.\] 
Thus, (i) implies (ii). Now, if (ii)
holds, then  
\[B_tI=B_tB_t^*B_tI=B_tIB_t^*B_t=(I\cap B_tB_t^*)B_t=(I
B_tB_t^*)B_t=IB_t,\]
which implies (iii). Clearly  $\mathcal{I}$ is  a right ideal, and it is 
apparent that it is also a left ideal if (iii) holds. Finally, suppose that 
$\mathcal{I}$ is an ideal in $\mathcal{B}$. Then 
\[B_tIB_t^*\subseteq
\mathcal{I}\cap B_e=I.\] 
\end{proof}

\begin{rk}
\label{intbinv}
If $J\lhd C^*(\mb)$ or $J\lhd C_r^*(\mb)$, then $J\cap B_e$ is a \mb-invariant ideal.
\end{rk}
\begin{proof}
In both cases 
$J_eB_t=J_t=B_tJ_e$, where  $J_t=J\cap B_t$,  for all $ t\in G$. It is 
clear that $J_t\supseteq J_eB_t$ and $J_t\supseteq B_tJ_e$. On the other hand,  
since $J_t$ is a Hilbert $C^*$ sub-bimodule of $B_t$, we have that $J_t=J_tJ_t^*J_t\subseteq J_eB_t\cap
B_tJ_e$.  
\end{proof}

\begin{prop}
\label{bundid}
Let $\mb$ be a Fell bundle over a discrete group $G$. The map $I\mapsto\mathcal{I}=(I_t)_{t\in G}$, where $I_t=IB_t$,  
is an isomorphism from the lattice
of $\mathcal{B}$-invariant ideals of $B_e$ onto that of ideals
of $\mathcal{B}$. Its inverse is given by $\mathcal{I}\mapsto
\mathcal{I}\cap B_e$.  
\end{prop}
\begin{proof}
Assume that $I$ is \mb-invariant. Then, by proposition \ref{binv}, $\mathcal{I}=(I_t)$ is an ideal in \mb, and the correspondence $I\mapsto\mathcal{I}$ is injective because $I_e=I$. Conversely, if $\mathcal{I}$ is an ideal of
$\mathcal{B}$,  let $I_t:=\mathcal{I}\cap B_t$, for all $t\in
G$. Since $\mathcal{I}$ is a Fell bundle and a right ideal of
$\mathcal{B}$, we have:  
\[ I_t=I_eI_t\subseteq I_eB_t\subseteq
\mathcal{I}\cap B_t=I_t.\] 
Then $I_t=I_eB_t$, and, analogously, $I_t=B_tI_e$. Thus, $I_e$ is a
$\mathcal{B}$-invariant ideal of $B_e$, and $\mathcal{I}=(I_eB_t)$. 
\par Finally, it is clear that both maps preserve inclusion, which implies they are lattice isomorphisms.  
\end{proof}
\begin{df}
 Recall that if $\alpha$ is a partial action of $G$ on a set $X$, then a set $S\subset X$ is said to be {\em $\alpha$-invariant}  if
 \[\alpha_t(S\cap \dom \alpha_t)=S\cap \dom \alpha_{t^{-1}}\text{, for all } t\in G.\]
 \end{df}

 \begin{prop}
\label{alphainv}
 Let $\mb$ be a Fell bundle over a discrete group $G$, and let $\hat{\alpha}$ be the partial action on $\hat{B_e}$ associated with 
 $\mb$. Then the map $J\mapsto X_J$ is an isomorphism from the lattice of \mb-invariant ideals in $B_e$ to that of open
 $\hat{\alpha}$-invariant sets in $\hat{B_e}$.
\end{prop}
 \begin{proof}
 Since it is well known that the correspondence  $J\mapsto X_J$ is a lattice isomorphism from $\cI(B_e)$ to the topology of $\hat{B_e}$, 
 the proof comes down to showing that an ideal $J$ in $B_e$ is \mb-invariant if and only if the open set $X_J$ is $\hat{\alpha}$-invariant.

 First assume that $J$ is \mb-invariant. If $\sigma \in X_J\cap X_{t^{-1}}$, then $\sigma|_{JD_{t^{-1}}}\neq 0$. Besides, $B_tJ=JB_t$ is a $D_tJ-JD_{t^{-1}}$ 
 imprimitivity bimodule, and it follows that  $\ind_{B_tJ}(\sigma|_{JD_{t^{-1}}}) \neq 0$. 
 
 On the other hand, if $\sigma$ is a representation on a Hilbert space $\mh_\sigma$, then the map $b_tj\otimes_{{\sss D_{t^{-1}}J}}h\mapsto b_tj\otimes_{D_{t^{-1}}}h$ extends to a unitary operator from $B_tJ\otimes_{{\sss D_{t^{-1}}J}}\mh_\sigma$ onto    $B_t\otimes_{D_{t^{-1}}}\mh_\sigma$ that intertwines $\ind_{B_t}(\sigma|_{D_{t^{-1}}})|_{D_tJ}$ and $\ind_{B_tJ}(\sigma|_{D_{t^{-1}}J})$. 
 This shows that 
 $\hat{\alpha}_t(\sigma)|_{J}\neq 0$, that is, that $\hat{\alpha}_t(\sigma)\in X_J$.

Assume now that  $X_J$ is $\hat{\alpha}$-invariant. Then
\[B_tJ=B_t D_{t^{-1}}J=h_{B_t}(D_{t^{-1}}J) B_t,\]
for all $t\in G$.

Now, since the Rieffel correspondence is a lattice isomorphism, 
\begin{align*}
 h_{B_t}(D_{t^{-1}}J)&=h_{B_t}\big(\bigcap\{ \ker\pi|_{D_{t^{-1}}}:\pi\in X^c_J\cap X_{t^{-1}}\}\big)\\
 &= \bigcap\{ h_{B_t}(\ker\pi|_{D_{t^{-1}}}):\pi\in X^c_J\cap X_{t^{-1}}\}\\
  &= \bigcap\{\ker\ind_{B_t}(\pi|_{D_{t^{-1}}}):\pi\in X^c_J\cap X_{t^{-1}}\}\\
 &=D_t\cap\bigcap\{ \ker \hat{\alpha}_t(\pi):\pi\in X^c_J\cap X_{t^{-1}}\}\\
&= D_t\cap \bigcap\{ \ker\pi:\pi\in X^c_J\cap X_{t}\}\\
&=D_tJ.
\end{align*}
Thus, $B_tJ=JD_tB_t=JB_t.$
 \end{proof}

\begin{df}
 Recall that a partial action $\alpha$ on a topological space $X$ is said 
 to be {\em minimal} if $X$ does not have $\alpha$-invariant open proper subsets.
\end{df}

\begin{cor}\label{cor:2.6}
Let $\mb=(B_t)_{t\in G}$ be a Fell bundle with associated partial
action~$\hat{\alpha}$. Consider the following statements: 
\begin{enumerate}

 \item $C^*_r(\mb)$ is simple.  
 \item The Fell bundle \mb  has no non-trivial ideals. 
 \item  $B_e$ has no non-trivial \mb-invariant ideals. 
 \item The partial action $\hat{\alpha}$ is minimal. 

\end{enumerate}
Then we have (i) $\Rightarrow$ (ii) $\iff$
(iii) $\iff$ (iv) and, if $\hat{\alpha}$ is topologically free, then
we also have (iv) $\Rightarrow$ (i), so in this case all the statements are
equivalent.    
\end{cor} 
\begin{proof}

Since all open proper subsets of $\hat{B_e}$  can be written as $X_J$ for some non-trivial ideal $J$ in $B_e$, 
Proposition \ref{bundid} and Proposition \ref{alphainv} show that ii),
iii) and iv) are equivalent. 

Assume now that $C^*_r(\mb)$ is
simple, and let $\mj\properideal\mb$. Then $C^*_r(\mj)\lhd
C^*_r(\mb)$ by \cite[3.2]{fa1}. Besides, since $\mj\neq\mb$,  we have that \[E(C^*_r(\mj))=\mj\cap B_e\neq B_e,\]
by Proposition \ref{bundid}. This implies that
$C^*_r(\mj)\neq C^*_r(\mb)$. Therefore, $C^*_r(\mj)=\{0\}$. We now have that $0\subseteq\mj\subseteq
C^*_r(\mj)=\{0\}$, hence $\mj=\{0\}$ and therefore i) implies~ii).

Suppose now that iv) holds and that $\hat{\alpha}$ is topologically
free. Let $J\properideal C^*_r(\mb)$, and set $J_e=J\cap B_e$.

By Remark \ref{intbinv}, $J_e$ is \mb-invariant. Now, by Proposition \ref{alphainv}, $X_{J_e}=\emptyset$, which implies that $J_e=\{0\}$.
It now follows from Corollary \ref{ibe}
that $J=\{0\}$, which implies that $C^*_r(\mb)$ is simple.
\end{proof}

\par Let $\ma=(A_t)_{t\in G}$ and $\mb=(B_t)_{t\in G}$
  be Fell bundles 
over a discrete group $G$. A map $\phi:\ma\to\mb$ is said to be a morphism if $\phi|_{A_t}:A_t\to B_t$ is linear for all $t\in G$, and
$\phi(aa')=\phi(a)\phi(a')$,  $\phi(a^*)=\phi(a)^*,$  for all $a,a'\in \ma$,  which implies that $\phi$ is norm decreasing. A morphism $\phi$ induces a homomorphism $\phi_c:C_c(\ma)\to C_c(\mb)$, given by $\phi_c(f)(t):=\phi(f(t))$.  The map $\phi_c$ is a $\norm{\ }_1$-continuous $*$-homomorphism, so it extends  to a
homomorphism of Banach $*$-algebras $\phi_1:L^1(\ma)\to L^1(\mb)$, and
hence to a $C^*$-algebra homomorphism $\phi_*:C^*(\ma)\to
C^*(\mb)$. Thus, we have a functor $(\ma \stackrel{\phi}{\to} \mb)\mapsto
(C^*(\ma)\stackrel{\phi_*}{\to}C^*(\mb))$, that turns out to be exact (\cite[3.1]{ae2}). 

If we now consider reduced $C^*$-algebras instead of
full $C^*$-algebras, we get another functor. In fact, suppose that 
$E_\ma:C^*(\ma)\to A_e$ is the canonical conditional expectation and that
$\Lambda_\ma:C^*(\ma)\to C^*_r(\ma)$ is the canonical
homomorphism. Since $\ker\Lambda_A=\{x\in C^*(\ma):
E_\ma(x^*x)=0\}$, and the diagram  
\[\xymatrix{C^*(\ma)\ar[r]^{\phi_*}\ar[d]_{E_\ma}&C^*(\mb)\ar[d]^{E_\mb}\\ A_e\ar[r]_{\phi|_{B_e}}&B_e}\]       
is  commutative, we have that $\phi_*(\ker\Lambda_A) \subseteq \ker\Lambda_\mb$. It follows
that there exists a unique homomorphism $\phi_r:C^*_r(\ma)\to
C^*_r(\mb)$ such
that \[\xymatrix{C^*(\ma)\ar[r]^{\phi_*}\ar[d]_{\Lambda_\ma}&C^*(\mb)\ar[d]^{\Lambda_\mb}\\ C^*_r(\ma)\ar[r]_{\phi_r}&C^*_r(\mb)}\] 
commutes. Thus, we have another functor $(\ma
\stackrel{\phi}{\to} \mb)\mapsto
(C^*_r(\ma)\stackrel{\phi_r}{\to}C^*_r(\mb))$. If $\phi$ is injective
or surjective, then so is $\phi_r$ (\cite[3.2]{fa1}). However, if we
consider the exact sequence of Fell bundles 
\[\xymatrix{0\ar[r]&\mi\ar[r]^i&\mb\ar[r]^p&\mb/\mi\ar[r]&0},\] 
where $\mi$ is an ideal in $\mb$, then the induced sequence
\[\xymatrix{0\ar[r]&C^*_r(\mathcal{I})\ar[r]^{i_r}&C^*_r(\mathcal{B})\ar[r]^{p_r}&C^*_r(\mathcal{B}/\mathcal{I})\ar[r]&0  
}\]
is not exact in general, because $C^*_r(\mi)$ does not necessarily agree with $\ker p_r$.

\par We remark that, since $\ker\Lambda_\ma=\{x\in
C^*(\ma): E_\ma(x^*x)=0\}$, we can define a map $C^*_r(\ma)\to A_e$
such that $\Lambda_A(x)\mapsto E_A(x)$, for all $ \Lambda_\ma(x)\in
C^*_r(\ma)$. This map is itself a faithful conditional
expectation (\cite[2.12]{examen}) with range $A_e$, which we will also
denote by $E_\ma$.  
\par Let $\cI(\mb)$ and $\cI(C^*_r(\mb))$ denote the lattice of ideals of the
  Fell bundle $\mb$ and in $C^*_r(\mb)$, respectively. Since for every $\mi\in \cI(\mb)$ we may
identify $C^*_r(\mi)$ with the closure of $C_c(\mi)$ in $C^*_r(\mb)$,
there is an order-preserving map $\mu:\cI(\mb)\to \cI(C^*_r(\mb))$
given by $\mu(\mi):=C^*_r(\mi)$.

We now consider the maps $\nu_1,\nu_2:\cI(C^*_r(\mb))\to \cI(\mb)$, given as follows. 
$\nu_1(J)$ is the ideal of $\mb$ corresponding to $J\cap B_e$ by
Proposition \ref{binv} (and Remark \ref{intbinv}). That is,
$\nu_1(J)=(J_t)_{t\in G}$, where $J_t=J\cap B_t$. Also, define
$\nu_2(J)$ to be the ideal of $\mb$ generated by $E_\mb(J)$. Then both $\nu_1$ and $\nu_2$ are left inverses for
$\mu$, which implies that $\mu$ is injective. However,  $\mu$ is not surjective in general. Clearly, a
necessary condition for $\mu$ to be onto is that $\nu_1=\nu_2$, that
is, that $J\cap B_e=E_\mb(J)$ for all $J\in \cI(C^*_r(\mb))$. 

\begin{df}\label{diaginv}(cf. \cite{xu})
  Let $\mathcal{B}=(B_t)_{t\in G}$ be a Fell bundle over a discrete
group $G$. An ideal $J$ of $C^*_r(\mb)$ is said to be \textit{diagonal
  invariant} if $E_\mb(J)\subseteq J$, that is, $E_\mb(J)=J\cap B_e$.    
\end{df}

\par In \cite{gs}, Giordano and Sierakowski thoroughly discussed the
correspondence $\mu$ above. In what follows, we generalize their methods and
results to the context of Fell bundles.  

\par Given an ideal $J$ of $C^*_r(\mb)$, let $\mj^{(1)}:=\nu_1(J)$ and $J^{(1)}:=\mu\nu_1(J)$, for $\mu$ and $\nu_1$ as above. Then $J^{(1)}\subseteq
J$, for it is the closure of the subset $C_c(\mj^{(1)})$ of
$J$.

 Similarly, we define $\mj^{(2)}:=\nu_2(J)$ and
$J^{(2)}:=\mu\nu_2(J)$. Then 
$\mj^{(2)}$ is the ideal of $\mb$ generated by $E_\mb(J)$, and
$J^{(2)}=C^*_r(\mj^{(2)})$. Note that the unit fiber of $\mj^{(2)}$ is
the invariant ideal of $B_e$ generated by the ideal $E_\mb(J)$ of
$B_e$. Since $E_\mb$ is the identity on $J\cap B_e$, it follows that  
$\mj^{(1)}\subseteq \mj^{(2)}$. Therefore, $J^{(1)}\subseteq
J\cap J^{(2)}$.       

\begin{df}(cf. \cite[Definition~3.1]{gs})
Let $\mathcal{B}=(B_t)_{t\in G}$ be a Fell bundle over the discrete
group $G$, and let $\mi=(I_t)_{t\in G}$ be an ideal of $\mb$. Then
\begin{enumerate}
 \item $\mathcal{B}$ is said to have the \textit{exactness property} at
   $\mi\lhd\mb$ if the sequence  
\[\xymatrix{0\ar[r]&C^*_r(\mathcal{I})\ar[r]^{i_r}&C^*_r(\mathcal{B})\ar[r]^{p_r}&C^*_r(\mathcal{B}/\mathcal{I})\ar[r]&0  
}\]
is exact.   
 \item $\mathcal{B}$ is said to have the \textit{intersection property} at $\mi$
   if the intersection of $B_e/{I_e}$ with
   any nonzero ideal in $C^*_r(\mathcal{B}/\mathcal{I})$ is also
   nonzero.    
\end{enumerate}
If $\mb$ has the exactness property at every ideal $\mi\in\cI(\mb)$,
we say that $\mb$ has the exactness property, and if it has the
intesection property at every ideal $\mi\in\cI(\mb)$, we say that
$\mathcal{B}$ has the \textit{residual intersection property}.   
\end{df}

\par In view of the previous definition, the second statement of
Corollary~\ref{ibe} could be restated in the following way: $\mb$ has the
intersection property whenever its associated partial action is
topologically free. More generally, we have: 
\begin{prop}\label{tfexact}
Let $\mb=(B_t)_{t\in G}$ be a Fell bundle over a discrete group
$G$. Suppose that $\mj=(J_t)_{t\in G}$ is an ideal of $\mb$, and let
$X:= \hat{B}_e\setminus X_{J_e}$. If the partial action of $\mb$ is 
topologically free on $X$, then $\mb$ has the
intersection property at the ideal $\mj$. 
\end{prop}
\begin{proof}
The unit fiber of the quotient bundle $\mb/\mj$ is $B_e/J_e$, whose
spectrum is homeomorphic to $\hat{B_e}\setminus X_{J_e}=X$. On the
other hand, it is readily checked that the partial action associated to
the Fell bundle $\mb/\mj$ agrees with the one induced by the partial
action of $\mb$. Now, by the commutativity of diagram (\ref{diag}) and
the fact that the partial action associated with $\mb$ is
topologically free on $X$, we conclude that the partial action
associated to $\mb/\mj$ is topologically free. Finally, we  apply part
(ii) in~\ref{ibe}.   
\end{proof} 
\begin{cor}\label{cor:tfexact}
If the partial action of the Fell bundle $\mb$ is topologically free
on every invariant closed subset of $\hat{B}_e$, then $\mb$ has the
residual intersection property.
\end{cor}

\begin{prop}\label{prop:exactness}
Let $\mathcal{B}=(B_t)_{t\in G}$ be a Fell bundle over a discrete
group $G$, and let $J\in\cI(C^*_r(\mb))$. 
\begin{enumerate}
 \item If $\mb$ has the exactness property at $\mj^{(2)}$, then
   $J\subseteq J^{(2)}$. If, in addition, $J$ is 
   diagonal invariant, then $J^{(1)}=J=J^{(2)}$.  
 \item If $\mb$ has the exactness property and the intersection
   property at $\mj^{(1)}$, then $J^{(1)}=J=J^{(2)}$.     
\end{enumerate}
\end{prop}
\begin{proof}
\par 
Let $\xymatrix{0\ar[r]&\mj^{(2)}\ar[r]^i&\mb\ar[r]^p&\mb/\mj^{(2)}\ar[r]&0}$
be the exact sequence associated with the ideal $\mj^{(2)}$ of
$\mb$, and suppose that $\mathcal{B}$ has the exactness property at
$\mj^{(2)}$.   
Then  the diagram   
\[\xymatrix{0\ar[r]&C^*_r(\mj^{(2)})\ar[r]^{i_r}\ar[d]^{E_{\mj^{(2)}}}&C^*_r(\mathcal{B})\ar[r]^{p_r}\ar[d]^{E_{\mb}}&C^*_r(\mathcal{B}/\mj^{(2)})\ar[r]\ar[d]^{E_{\mb/\mj^{(2)}}}&0\\    
0\ar[r]&\mj^{(2)}\cap B_e\ar[r]_i&B_e\ar[r]_p&B_e/(\mj^{(2)}\cap B_e)\ar[r]&0 
}\]
is commutative and has exact rows.
If $x\in J^+$, then $E_{\mb}(x)\in \mj^{(2)}\cap B_e^+$, which implies that
$E_{\mb/\mj^{(2)}}p_r(x)=0$. Since $p_r(x)\in C^*_r(\mb/\mj^{(2)})^+$ and $E_{\mb/\mj^{(2)}}$
is faithful, then $p_r(x)=0$. Then $x\in C^*_r(\mj^{(2)})$, because of the 
exactness of the first row at $C^*_r(\mb)$. This shows that
$J\subseteq J^{(2)}$. Since the inclusion
  $J^{(1)}\subseteq J$  always holds, and the definition of
  diagonal invariance requires precisely that $\mj^{(1)}=\mj^{(2)}$,
  which implies that $J^{(1)}=J^{(2)}$, we conclude that
  $J^{(1)}=J=J^{(2)}$.  
\par Suppose now that $\mathcal{B}$ has both the exactness  
 and the residual intersection properties at $\mj^{(1)}$. Let 
 $q:\mb\to \mb/\mj^{(1)}$ be the quotient map. In order to
   prove that $J^{(1)}=J=J^{(2)}$, it suffices to show that $J^{(1)}=J$, for
   in this case we have that $E(J)\subseteq J^{(1)}$, and, consequently, that
   $J^{(2)}=J^{(1)}$. In other words, we have to show that
 $q_r(J)=\{0\}$. Since 
 $\mb$ is exact at $\mj^{(1)}$, we have $\ker q_r=J^{(1)}$. Let
 $\bar{q}_r:C^*_r(\mb)/J^{(1)}\to C^*_r(\mb/\mj^{(1)})$ be the
 isomorphism induced by $q_r$. Since $\mb$ has the intersection
 property at $\mj^{(1)}$, in order to prove that $q_r(J)=\{0\}$, it suffices
 to show that $q_r(J)\cap B_e/(J\cap B_e)=\{0\}$, or, equivalently, that
 \begin{equation}
\label{zero}
J/J^{(1)}\cap
 (B_e+J^{(1)})/J^{(1)}=\{0\},
\end{equation}
since 
\[J/J^{(1)}\cap
 (B_e+J^{(1)})/J^{(1)}=\bar{q}_r^{-1}(q_r(J)\cap B_e/(J\cap B_e)).\] 
Let $x\in J$, and $b\in B_e$ be such that $x+J^{(1)}=b+J^{(1)}\in J/J^{(1)}\cap
 (B_e+J^{(1)})/J^{(1)}$.  Then $x-b\in
 J^{(1)}\subseteq J$, which implies that $b\in J\cap
 B_e\subseteq J^{(1)}$ and $x\in J^{(1)}$, so (\ref{zero}) holds,  and (ii) follows.
\end{proof}

\begin{lem}\label{latquot}
If the map $\mu:\mi(\mb)\to\mi(C^*_r(\mb))$ given by $\mi\mapsto
C^*_r(\mi)$ is a lattice isomorphism and $\mb$ has the exactness
property at $\mj\in \mi(\mb)$, then
$\mu_{\mj}:\mi(\mb/\mj)\to \mi(C^*_r(\mb/\mj))$ given by $\mi\mapsto
C^*_r(\mi/\mj)$ is also a lattice isomorphism.  
\end{lem}
\begin{proof}
Let $\mi_\mj:=\{\mi\in\mi(\mb):\mj\subseteq\mi\}$ and
$\mi_{\mu(\mj)}:=\{I\in\mi(C^*_r(\mb)):\mu(\mj)\subseteq I\}$. Then
the restriction of $\mu$ to $\mi_\mj$ gives rise to an isomorphism
between $\mi_\mj$ and $\mi_{\mu(\mj)}$. On the other hand, the map
$\eta_1:\mi\mapsto\mi/\mj$ is an isomorphism from $\mi_\mj$ onto 
$\mi(\mb/\mj)$, as is the map $\eta_2: I\mapsto I/C^*_r(\mj)$
from $\mi_{\mu(\mj)}$ onto $\mi(C^*_r(\mb)/C^*_r(\mj))$. Moreover,
since $\mb$ is exact at $\mj$, the quotient 
map $p:\mb\to\mb/\mj$ induces an isomorphism 
$\bar{p}_r:C^*_r(\mb)/C^*_r(\mj)\to C^*_r(\mb/\mj)$, which in turn
induces an obvious lattice isomorphism 
$\eta_3:\mi(C^*_r(\mb)/C^*_r(\mj))\to\mi(C^*_r(\mb/\mi))$. Then 
$\mu_\mj$ is an isomorphism, because
$\mu_\mj=\eta_3\eta_2\mu|_{\mi_\mj}\eta_1^{-1}$.       
\end{proof}

\begin{thm}\label{thm:gs}
Let $\mathcal{B}=(B_t)_{t\in G}$ be a Fell bundle over a discrete
group $G$. Let $\mu: \cI(\mb)\longrightarrow \cI(C_r^*(\mb))$ be the lattice homomorphism given by
$\mu(\cI)=C^*_r(\cI)$.  Then the following
statements are equivalent:
\begin{enumerate}
 \item The map $\mu$ is an isomorphism of lattices.
 \item $\mathcal{B}$ has the exactness property and every
   $J\in\cI(C^*_r(\mb))$ is diagonal invariant. 
 \item $\mathcal{B}$ has the exactness and residual intersection
properties. 
\end{enumerate}   
\end{thm}
\begin{proof}
It follows from Proposition~\ref{prop:exactness} that either statement (ii) or (iii) implies (i). Suppose that
$\mu$ is a lattice isomorphism. Then any ideal of $C^*_r(\mb)$ is of
the form $C^*_r(\mi)$, and therefore is diagonal invariant. Recall
from the comments preceeding Definition \ref{diaginv}   
that the inverse of $\mu$ is given by $J\mapsto J\cap B_e$. To show 
that (i) implies (ii) we have to prove that $\mb$ has the exactness
property at any ideal $\mi=(I_t)_{t\in G}$ of $\mb$. The quotient map 
$p:\mb\to\mb/\mi$ induces a surjective homomorphism $p_r:C^*_r(\mb)\to
C^*_r(\mb/\mi)$, whose kernel contains $C^*_r(\mi)$. Then
$I_e=E_\mb(C^*_r(\mi))\subseteq E_\mb(\ker(p_r))=\ker(p_r)\cap B_e$, the last equation following from the diagonal invariance of
$\ker(p_r)$. But $\ker(p_r)\cap B_e=\ker(p|_{B_e})=I_e=C^*_r(\mi)\cap
B_e$. Then $\ker(p_r)=C^*_r(\mi)$. 
\par To conclude that (i) also implies (iii) we have to show that
$\mb$ has the residual intersection property. So pick an element
$\mj=(J_t)_{t\in G}\in\mi(\mb)$, and suppose that $I\lhd
C^*_r(\mb/\mi)$ is such that $I\cap \frac{B_e}{J_e}=\{0\}$. By
Lemma~\ref{latquot} there is a unique $\mi=(I_t)_{t\in G}\lhd\mb$ such
that $\mj\subseteq\mi$ and $I=C^*_r(\mi/\mj)$. Then
\[\{0\}=I\cap\frac{B_e}{J_e}=\frac{I_e\cap B_e}{J_e}.\] 
That is, $J_e=I_e$. Since, by \ref{bundid}, this implies that $\mi=\mj$, it follows that $I=\{0\}$.  
\end{proof}

\begin{cor}
\label{320}
Let $\mb=(B_t)_{t\in G}$ be a Fell bundle over a discrete group $G$. 
Then the correspondences $\mj\mapsto  C^*(\mj)$ and $\mj\mapsto
C_r^*(\mj)$ are injective lattice homomorphisms from the lattice of ideals 
in $\mb$ to the lattices $\cI(C^*(\mb))$ and $\cI(C_r^*(\mb))$ of
ideals in $C^*(\mb)$ and $C_r^*(\mb)$ respectively. If $\mb$ has the
exactness property and its associated partial action is topologically  
free on every $\hat{\alpha}$-invariant closed subset of $\hat{B_e}$,
then $\mi(\mb)\to\mi(C^*_r(\mb))$ is a lattice isomorphism. 
\end{cor}
\begin{proof}
Let $\mj=(J_t)_{t\in J}$ be an ideal in $\mb$. By \cite[3.1]{ae2}, $\overline{C_c(\mj)}=C^*(\mj)\lhd
C^*(\mb)$, where $\overline{C_c(\mj)}$ is the closure of $C_c(\mj)$ in $C^*(\mb)$.
It follows that $B_e\cap C^*(\mj)=\mj\cap B_e$, which 
takes care of the injectivity, in view of
Proposition~\ref{bundid}.
 The rest of the proof follows immediately from \ref{thm:gs} and
\ref{cor:tfexact}.  
\end{proof}

\begin{ex}{\rm {\bf Ideal structure of Quantum Heisenberg Manifolds.}
 The family $\{D^c_{\mu,\nu}: c\in \ZZ, c>0, \mu, \nu\in \TT\}$ of Quantum Heisenberg Manifolds was constructed in \cite{rfhm} as a 
deformation of the Heisenberg manifold $M_c$ for a positive integer $c$. The $C^*$-algebra $\dmn$ was shown in \cite{aee}  to be the 
crossed product of $C(\TT^2)$ by a Hilbert $C^*$-bimodule $X^c_{\mu\nu}$, where $\TT$ denotes the unit circle.
Since $X^c_{\mu\nu}$ is full in both the left and the right, $\hat{\alpha}$ turns out to be a homeomorphism,  that was 
shown in \cite{ae1} (see also Example \ref{cphcb}) to be given by
\[\hat{\alpha}(x,y)=(x+2\mu, y+2\nu),\text{ for all } (x,y)\in \TT^2.\]
Let $G_{\mu\nu}$ denote the abelian free group $G_{\mu\nu}= \ZZ+2\mu\ZZ+2\nu\ZZ$. Rieffel showed in \cite[6.2]{rfhm} that $\dmn$ is simple if and only if
rank $G_{\mu\nu}=3$. On the other hand, when rank $G_{\mu\nu}=1$ the $C^*$-algebra $\dmn$ is Morita equivalent to the commutative 
$C^*$-algebra $C(M_c)$ (\cite[2.8]{meq}), 
and, consequently, has the same ideal structure.
We now discuss the case in which rank $G_{\mu\nu}=2$. First note that the action $\hat{\alpha}$ is free in that case. In fact,
$\hat{\alpha}_n(x,y)=(x,y)$ if and only if $2n\mu$ and $2n\nu$ are integers, which implies that $n=0$ or  rank $G_{\mu\nu}=1$.

Besides, $C(\TT^2)\rtimes X^c_{\mu\nu}$  has the exactness property by \cite[3.1]{ae2}, because it is the cross-sectional $C^*$-algebra of a Fell bundle \mb over 
the amenable group $\ZZ$. Thus, we are under the assumptions of Lemma
\ref{320}, and  
there is a lattice isomorphism between $\cI(\dmn)$ and the lattice of $\hat{\alpha}$-invariant open sets of the two-torus.

 }
\end{ex}

\section{Fell bundles with commutative unit fiber}

Throughout this section we will assume that the unit fiber of the Fell bundle \mb is commutative. That is, $B_e=C_0(X)$, for some locally compact Hausdorff space $X$. We will make use of the identifications and facts we established in Example \ref{cuf}.
 Let $j_t:B_t\longrightarrow \ell^2(\mb)$ be the inclusion map described in (\ref{jayt}). Exel proved in \cite{examen} that, for any  $c\in C^*_r(\mb)$ and $t\in B_t$, 
there is a unique element $\hat{c}(t)\in B_t$, called the
Fourier coefficient of $c$ corresponding to~$t$, such that
\[j_t^*cj_e(a)=\hat{c}(t)a, \ \forall a\in B_e.\]  
He also showed that
$c=0$ if and only if  $\hat{c}=0$ (\cite[2.6, 2.7, 2.12]{examen}).  

\begin{lem}\label{lem:commfourier}
Let $a\in B_e$ and $c\in C^*_r(\mb)$. Then $\widehat{ac}=a\hat{c}$ and
$\widehat{ca}=\hat{c}a$. 

Consequently, $c$ commutes with $a$ if and
only if $a\hat{c}(t)=\hat{c}(t)a$ for all $t\in G$.   
\end{lem}
\begin{proof}
Note that $\Lambda_aj_e(a')=j_e(aa')$, $\forall a'\in B_e$. Then 
\[\widehat{ca}(t)a' 
=j_t^*caj_e(a')
=j_t^*cj_e(aa')
=j_t^*cj_e(a)a'
=\hat{c}(t)aa',
\]
and it follows that $\widehat{ca}=\hat{c}a$. 

On the other hand,
as it is easily checked, $j_t^*\Lambda_a(\xi)=a\xi(t)$, for all $
\xi\in\ell^2(\mb)$. Therefore, if $a'\in B_e$: 
\[
\widehat{ac}(t)a'
=j_t^*acj_e(a')
=aj_t^*cj_e(a')
=a\hat{c}(t)a',
\]
which shows that $\widehat{ac}(t)=a\hat{c}(t)$. 
The last statement follows from the
first one and from the fact that $ac=ca$ if and only if $\widehat{ac-ca}=0$.
\end{proof}

\begin{lem}\label{lem:commbundle}
Let $b_t\in B_t$, and 
\[F_t=\{x\in X_{t^{-1}}: \hat{\alpha}_t(x)= x\}.\]  
Then $b_t\in B_e'$ if and only if 
$b_t(x)=0$ for all $ x\notin F_t$.
\end{lem}
\begin{proof}
Since $ab_t=b_ta$ if and only if $(ab_t-b_ta)(x)=0$  for all $x\in X,$
we have that $b_t\in B_e'$ if and only if
$b_t(x)a(\hat{\alpha}_t(x))=b_t(x)a(x)$ for all $ x\in 
X_{t^{-1}}$ and $ a\in B_e$. Thus $b_t\in B_e'$ if
$b_t(x)=0$ for all $x\notin F_t$. 

Conversely, if $b_t\in
B_e'$, and $x\in X_{t^{-1}}\setminus
F_t$, we can pick an element $a\in B_e$ such that $a(x)\neq
0=a(\hat{\alpha}_t(x))$. Then $b_t(x)a(x)=0$, which shows that $b_t(x)=0$. 
\end{proof}

\par  Zeller-Meier showed that if $\alpha$ is an action of a discrete group $G$ 
on a commutative
$C^*$-algebra $A$, then $A$ is a maximal commutative $C^*$-subalgebra of the reduced crossed product
$A\rtimes_{\alpha,r}G$ if and only if $\alpha$ is topologically free
on $\hat{A}$ (\cite[Proposition~4.14]{zm}). The previous 
results allow us to generalize that result in the following
way:

\begin{prop}\label{cor:commreduced}
Let $B_e'$ be the commutant of $B_e$ in $C^*_r(\mb)$. Then $B_e'=B_e$
if and only if $\hat{\alpha}$ is topologically free.      
\end{prop}
\begin{proof}
Let $c\in C^*_r(\mb)$. By \ref{lem:commfourier} and
\ref{lem:commbundle} we have $c\in B_e'$ if and only if
$\hat{c}(t)=0$ outside $F_t$, $\forall t\in G$. Then if for all $t\neq 
e$ the interior of $F_t$ is empty, we have $\hat{c}(t)=0$, so $c\in
B_e$, and therefore $B_e'=B_e$. On the other hand, if there exists
$t\neq e$ such that $F_t$ has a non empty interior, then
there exists $a\in D_{t^{-1}}$, $a\neq 0$, such that
$a(x)=0$ $\forall x\notin F_t$. Since $B_ta\neq 0$, there   
exists $b'_t\in B_t$ such that $0\neq b'_ta=:b_t\in B_t$. Now
$b_t(x)=0$ $\forall x\notin F_t$, and therefore $b_t\in
B_e'\setminus B_e$.      
\end{proof}

\begin{cor}\label{cor:commreduced2}
The partial action $\hat{\alpha}$ is topologically free if and only if $B_e$
is a maximal commutative $C^*$-subalgebra of $C^*_r(\mb)$ (and, consequently, it is 
a Cartan subalgebra of $C^*_r(\mb)$).    
\end{cor}

\subsection{The case of partial crossed products.}

 \par We will consider next a partial action on a
 commutative $C^*$-algebra $A=C_0(X)$, where $X$ a locally compact
 Hausdorff space. It is clear from Example \ref{pcp} that in this case the partial action $\hat{\alpha}$ associated to the Fell 
 bundle agrees with $\alpha$ when $X$ is identified in the usual way with $\hat{A}$. 
 In what follows we will write $\alpha$ to denote either one.

\begin{thm}\label{thm:tfpactions}
 Suppose that $\alpha$ is  a partial action of a discrete group $G$ on a
 commutative $C^*$-algebra. Consider the following
 statements: 
\begin{enumerate}
 \item $A$ is a maximal commutative $C^*$-subalgebra of
   $A\rtimes_{\alpha,r}G$. 
 \item $\alpha$ is a topologically free. 
 \item If $I$ is an ideal in $A\rtimes_{\alpha}G$ with $A\cap I=\{0\}$,
   then $I\subseteq \ker\Lambda$, where
   $\Lambda:A\rtimes_{\alpha}G\to A\rtimes_{\alpha,r}G$ is the
   canonical map.  
 \item If $I$ is a non-zero ideal of $A\rtimes_{\alpha,r}G$, then
   $A\cap I\neq \{0\}$. 
 \item If a representation $\phi:A \rtimes_{\alpha,r}G\to B(H)$ is
   faithful when restricted to $A$, then $\phi$ is faithful.     
 \end{enumerate}
Then we have that $(i)\iff (ii)\iff  (iii)\Rightarrow (iv)\iff (v)$. 
\end{thm}
\begin{proof} 
Corollary \ref{cor:commreduced2} shows that (i) and (ii) are equivalent. Besides, Corollary \ref{ibe} proves that
(ii) implies (iii), and its proof shows that (iii) implies (iv). Since (iv) and (v) are obviously equivalent, 
we are left with the proof of the fact that (iii) implies (ii).  
We will adapt to our setting the proof for global actions in \cite[Theorem~2]{ar}, which in turn
   essentially follows \cite{kt}.
   Suppose (iii) holds. Let $X$ be a locally compact Hausdorff topological space such that $A=C_0(X)$.

   Given $x\in
   X$, let $o(x)$ denote the $\alpha$-orbit of $x$: 
   $o(x):=\{\alpha_t(x):\,t \textrm{ such that }x\in X_{t^{-1}}\}$. Let 
   $H^x:=\ell^2(o(x))$ with its canonical orthonormal basis
   $\{e_y:y\in o(x)\}$. Consider $v^x:G\to B(H^x)$ defined by 
   \[v_t^x(e_y)=\begin{cases}e_{\alpha_t(y)}&\textrm{if
   }y\in X_{t^{-1}}\\ 0&\textrm{otherwise.}\end{cases}\]

   Thus $v_t^x$
   is a partial isometry with initial space $\ell^2(o(x)\cap
   X_{t^{-1}}))$ and final space $\ell^2(o(x)\cap X_{t})$.

   We claim
   that $v^x$ is a partial representation of $G$. Let us first note 
   that $(v_t^x)^*=v_{t^{-1}}^x$, since
  \[\pr{v_t^x(e_y)}{e_z}=\begin{cases}1&\textrm{ if } y\in X_{t^{-1}} \textrm{ and } z\in X_t\\
                          0&\textrm{otherwise}\end{cases}
                         =\pr{e_y}{v_{t^{-1}}^x(e_z)}.\]
We next show that
   \[v_r^xv_s^xv_{s^{-1}}^x(e_y)=v_{rs}^xv_{s^{-1}}^x(e_y), \textrm{ for all } r,s\in G,\  y\in o(x).\]
In fact, we have on the one hand that
  \[ v_r^xv_s^xv_{s^{-1}}^x(e_y)=\begin{cases} e_{\alpha_r(y)}&\textrm{ if } y\in X_s\cap X_{r^{-1}}\\
                          0&\textrm{otherwise.}\end{cases}\]
 On the other hand,
             \[v_{rs}^xv_{s^{-1}}^x(e_y) =\begin{cases} 0&\textrm{ if } y\not\in X_s\cap \alpha_s(X_{s^{-1}r^{-1}}\cap X_{s^{-1}})=
             X_{r^{-1}}\cap X_s\\
                          e_{\alpha_r(y)}&\textrm{otherwise.}\end{cases}\]             
   
\par We now define the representation $\pi^x:A\to
   B(H^x)$ by $\pi^x(a)(e_y)=a(y)e_y$, for all $ a\in
   A$ and  $y\in o(x)$.

   We claim that the pair $(\pi^x,v^x)$ is a covariant
representation of the system $(A,\alpha)$. In fact, if $a\in
C_0(X_{t^{-1}})$, $y\in o(x)$: 
\begin{gather*}
\pi^x(\alpha_t(a))(e_y)
=\alpha_t(a)(y)e_y
=\begin{cases} a(\alpha_{t^{-1}}(y))e_y&\textrm{if }y\in
X_t\\
0&\textrm{otherwise.}
 \end{cases}
\end{gather*} 
On the other hand, $v_t^x\pi^x(a)v_{t^{-1}}^x(e_y)=0$ if $y\notin
X_t$, and if $y\in X_t$:  
\begin{gather*}
v_t^x\pi^x(a)v_{t^{-1}}^x(e_y)
=v_t^x(a(\alpha_{t^{-1}}(y))e_{\alpha_{t^{-1}}(y)})
=a(\alpha_{t^{-1}}(y))e_{y}.
\end{gather*} 
\par Let $\rho^x:A\rtimes_\alpha G\to B(H^x)$ be the integrated form 
$\rho^x=\pi^x\rtimes v^x$ of the covariant representation
$(\pi^x,v^x)$. If $I=\bigcap_{x\in X}\ker\rho^x$, then
$I\cap A=0$, since if $a\in A$ and $\rho^x(a)=0$ for all $x\in X$,
then 
\[0=\rho^x(a)(e_y)=a(y)e_y,\ \forall x\in X, y\in o(x),\]
which shows that $a=0$. Since we are assuming that (iii) holds,  $I\subseteq
\ker\Lambda$. 

Let $t\neq e$ and $a\in A$ be such that
$\textrm{supp}(a)\subseteq \{x\in X\cap
X_{t^{-1}}:\alpha_t(x)=x\}$. Then we have, for $x\in X$, $y\in o(x)$:  
\begin{itemize}
 \item[$\bullet$] if $y\in\textrm{supp}(a)$ then
   $\alpha_t(y)=y$, and 
\[\rho^x(a\delta_e-a\delta_t)(e_y)
=a(y)e_y-a(\alpha_t(y))e_{\alpha_t(y)}=0.\]
 \item[$\bullet$] if $y\notin\textrm{supp}(a)$ then
   $\alpha_t(y)\notin\textrm{supp}(a)$, and therefore we have 
\[\rho^x(a\delta_e-a\delta_t)(e_y)
=a(y)e_y-a(\alpha_t(y))e_{\alpha_t(y)}=0.\]  
From the computations above we conclude that $a\delta_e-a\delta_t\in
I$. Therefore $a\delta_e-a\delta_t\in\ker\Lambda$. Then
$a=E(a\delta_e-a\delta_t)=0$, from which it follows that the set $\{x\in X\cap
X_{t^{-1}}:\alpha_t(x)=x\}$ has empty interior. 
\end{itemize}
\end{proof}

\begin{thebibliography}{999}

\bibitem{ae1} B. Abadie, R. Exel, {\it Hilbert $C^*$-bimodules over
    commutative $C^*$-algebras and an isomorphism condition for quantum
    Heisenberg manifolds}, Reviews in Mathematical Physics, Vol. {\bf
    9}, No. 4 (1997), 411-423.
    
\bibitem{aee} B. Abadie,  S. Eilers, R. Exel, {\it Morita equivalence for crossed products by Hilbert C*-bimodules},
Trans. Amer. Math. Soc. \textbf{350} (1998) (8) 3043-3054.

\bibitem{ae2} B. Abadie, R. Exel, {\it Deformation Quantization via
    Fell Bundles}, Math. Scand. \textbf{89} (2001), 135-160.  
\bibitem{meq}  B. Abadie,  {\it Morita equivalence for quantum Heisenberg manifolds},
Proc.  Amer. Math. Soc. \textbf{133} (2005) No 12,  3515-3523.

\bibitem{fa1} F. Abadie, {\it Enveloping Actions and Takai Duality for
    Partial Actions}, J.~Funct. Anal. \textbf{197} (2003), 14-67. 
    \bibitem{fa2} F. Abadie, {\it On Partial Actions and Groupoids},
  Proc. Amer. Math. Soc. \textbf{132} (2004), (4), 1037-1047.  
\bibitem{antleb} A. B. Antonevich, A. V. Lebedev, {\it Functional Differential Equations: I, C*-theory}, Longman Scientific 
\& Technical, Harlow, Essex, England 1994.  
  
  \bibitem{ar} R.~J.~Archbold, J.~S.~Spielberg, \textit{Topologically
  free actions and ideals in discrete $C^*$-dynamical systems}, 
  Proc. Edinburgh Math. Soc. (2) \textbf{37} (1994), no. 1, 119-124.   
  \bibitem{eh} E.~G.~Effros, F.~Hahn, \textit{Locally compact
    transformation groups and $C^*$-algebras},
  Mem. Amer. Math. Soc. \textbf{75} (1967). 
\bibitem{examen} R.~Exel, \textit{Amenability for Fell bundles},
  Journal fur die Reine Angew. Math. \textbf{492} (1997), 41-74. 
\bibitem{exex} R.~Exel, \textit{Exact groups and Fell bundles},
  Math. Ann. \textbf{323} (2002), 259-266.  
\bibitem{elq} R. Exel, M. Laca, J. Quigg, {\it Partial Dynamical
    Systems and $C^*$-algebras generated by Partial Isometries},
    Journal of Operator Theory \textbf{47} (1), (2002)169-186.
\bibitem{fd} J. M. Fell, R. S. Doran, {\it Representations of
    *-algebras, locally compact groups, and Banach *-algebraic
    bundles}, Pure and Applied Mathematics vol. {\bf 125} and {\bf
    126}, Academic Press, 1988.  
\bibitem{gs} T Giordano, A Sierakowski, \textit{Purely infinite
  partial crossed products} J. Funct. Anal. \textbf{266} (2014),
  no. 9, 5733-5764.   
\bibitem{kt} S.~Kawamura, J~Tomiyama, \textit{Properties of
  topological dynamical systems and corresponding $C^∗$-algebras},
  Tokyo J.~Math. \textbf{13} (1990), no. 2, 251-257.  
\bibitem{kwa} B.K. Kwa\'sniewski, \textit{Topological Freeness for Hilbert Bimodules},
  Israel Journal of Mathematics \textbf{199}, no. 2, (2014), 641-650.
\bibitem{leb} A.~V.~Lebedev, \textit{Topologically free partial
  actions and faithful representations of crossed products},
   Funct. Anal. Appl. \textbf{39} (2005), no. 3, 207–214. 
\bibitem{od} D.~P.~O'Donovan, \textit{Weighted shifts and covariance
    algebras}, Trans. Amer. Math. Soc. \textbf{208} (1975), 1-25. 
 \bibitem{rw} I. Raeburn, D. P. Williams, {\it Morita Equivalence and   
        Continuous--Trace $C^*$-algebras}, Math. Surveys and
        Monographs, Volume {\bf 60}, Amer. Math. Society, 1998.  
        
\bibitem{rfhm} M. Rieffel, {\it Deformation quantization of Heisenberg manifolds},
Comm. Math. Phys. {\bf 122}, (1989) 531-562.

\bibitem{xu} Q.~Xu, X.~Zhang, \textit{Diagonal invariant ideals of
  topologically graded C∗-algebras},  Chinese Quart. J. Math. 20
  (2005), no. 4, 338–341. 
 
\bibitem{zm} G.~Zeller-Meier, \textit{Produits crois\'es d' une
    $C^*$-alg\`ebre par un groupe d' automorphismes}, J.~Math. Pures
  et Appl., iX S\'er. \textbf{47} (1968), 101-239. 
\end{thebibliography}
\end{document}